\documentclass[a4paper]{article}
\pdfoutput=1

\usepackage{amsmath}
\usepackage{amsfonts, amssymb, amsthm}
\usepackage{mathtools}
\usepackage{microtype}
\usepackage{color}
\usepackage[a4paper]{geometry}
\usepackage{xcolor}
\usepackage[colorlinks = true,
            linkcolor = blue,
            urlcolor  = blue,
            citecolor = blue,
            anchorcolor = blue]{hyperref}

\newcommand{\R}{{\mathbb{R}}}
\newcommand{\Lip}{L}
\DeclareMathOperator{\diag}{diag}

\newtheorem{theorem}{Theorem}[section]

\newtheorem{lemma}[theorem]{Lemma}

\theoremstyle{definition}

\theoremstyle{remark}
\newtheorem{remark}[theorem]{Remark}

\theoremstyle{definition}

\numberwithin{equation}{section}

\begin{document}

\title{A note on the optimal convergence rate of descent methods with fixed step sizes for smooth strongly convex functions}
\author{Andr\'e Uschmajew\thanks{Max Planck Institute for Mathematics in the Sciences, 04103
Leipzig, Germany.} \and Bart Vandereycken\thanks{Section of Mathematics, University of Geneva, 1211 Geneva, Switzerland. This work was supported by the SNSF under research project 192129.}}
\date{}

\maketitle

\begin{abstract}
Based on a result by Taylor, Hendrickx, and Glineur ({\em J.~Optim.~Theory~Appl.}, 178(2):455--476, 2018) on the attainable convergence rate of gradient descent for smooth and strongly convex functions in terms of function values, an elementary convergence analysis for general descent methods with fixed step sizes is presented. It covers general variable metric methods, gradient related search directions under angle and scaling conditions, as well as inexact gradient methods. In all cases, optimal rates are obtained.
\end{abstract}

\section{Introduction}

An $L$-smooth and $\mu$-strongly convex function $f \colon \R^n \to \R$ is  characterized by the two properties
\begin{equation*}\label{eq: Lipschitz condition}
\| \nabla f(x) - \nabla f(y) \| \le \Lip \| x - y \|
\end{equation*}
and
\begin{equation*}\label{eq: strong convexity}
 f(x) \ge f(y) + \langle \nabla f(y), x - y \rangle + \frac{\mu}{2} \| x - y \|^2 
\end{equation*}
for some constants $0 < \mu \le L$ and all $x,y \in \R^n$. Here, $\langle \ , \  \rangle$ can be any inner product on $\R^n$ with corresponding norm $\| \cdot \|$, and $\nabla f$ denotes the gradient with respect to this inner product. Note that the  constants $\mu$ and $L$ depend on the chosen inner product. The class of such functions plays a main role in the convergence theory of the gradient method and related descent methods for finding the unique global minimum $x^*$ of a given $f$. The update rule of the gradient method is
\[
 x^+ = x - h \nabla f(x),
\]
where $h > 0$ is a step size which may depend on the current point $x$. It is well known that the fixed step size
\[
h = \frac{2}{L + \mu}
\]
achieves the optimal error reduction
\begin{equation}\label{eq: rate for norm}
\| x^+ - x^* \|^2 \le  \left(\frac{\kappa_f - 1}{\kappa_f + 1} \right)^2 \| x - x^* \|^2, \quad \kappa_f = \frac{L}{\mu},
\end{equation}
per step, which inductively implies the convergence of the method to $x^*$. We refer to~\cite[Theorem~2.1.15]{Nesterov2004} for details.

In a more general setting of proximal gradient methods, it has recently been shown by Taylor, Hendrickx, and Glineur~\cite[Theorem~3.3 with $h=0$]{Taylor2018} that the same rate is also valid for the error in function value. Specifically, for any 
\begin{equation}\label{eq: feasible step sizes}
 0 \le h \le \frac{2}{L + \mu}
\end{equation}
it holds that
\begin{equation}\label{eq: function values general h}
 f(x^+) - f(x^*) \le (1 - h\mu)^2 (f(x) - f(x^*)).
\end{equation}
Moreover, for $\frac{2}{L+\mu} \le h < \frac{2}{L}$ one has $f(x^+) - f(x^*) \le (hL - 1)^2 (f(x) - f(x^*))$. This automatically follows from~\eqref{eq: feasible step sizes} and~\eqref{eq: function values general h} by using  a weaker strong convexity bound $0 < \mu' \le \mu$ satisfying $h = \frac{2}{L + \mu'}$ and noting that $1 - h \mu' = hL - 1$. The optimal choice in the estimates is $h = 2/(L + \mu)$ and leads to
\begin{equation}\label{eq: estimate for function values}
 f(x^+) - f(x^*) \le  \left(\frac{\kappa_f - 1}{\kappa_f + 1} \right)^2 (f(x) - f(x^*)).
\end{equation}
This estimate for one step of the method is highly nontrivial. Obviously, it implies the same inequality for the gradient descent method with exact line search (when the left side is minimized over all $h$), which has been obtained earlier in~\cite{deKlerk2017}. Moreover, this estimate is known to be optimal in the class of $L$-smooth and $\mu$-strongly convex functions. In fact, it is already optimal for quadratic functions in that class; see, e.g.,~\cite[Example~1.3]{deKlerk2017}.

Of course, in many applications the difference $f(x) - f(x^*)$ is a natural error measure by itself. For example, for strongly convex quadratic functions it is proportional to the squared energy norm of the quadratic form. In general, for an $L$-smooth and $\mu$-convex function we always have
\[
 \frac{\mu}{2} \| x - x^* \|^2 \le f(x) - f(x^*) \le \frac{L}{2} \| x - x^* \|^2,
\]
which clearly shows that $f(x_{\ell}) - f(x^*) \to 0$ for an iterative method implies $\| x_\ell - x^* \| \to 0$ for $\ell \to \infty$. Moreover, both error measures will exhibit the same $R$-linear convergence rate. The novelty of the estimate~\eqref{eq: estimate for function values} is that one also has an optimal $Q$-linear rate for the function values, both for fixed step sizes and exact line search. (We refer to~\cite{ortegaIterativeSolutionNonlinear1970} for the definitions of $R$- and $Q$-linear rate.) However, compared to~\eqref{eq: rate for norm} an estimate like~\eqref{eq: estimate for function values} is ``more intrinsic'', because the chosen inner product in $\R^n$ enters only via the constants $\mu$ and $L$. In this short note, we illustrate this advantage by showing that~\eqref{eq: estimate for function values} allows for a rather clean analysis of general variable metric methods, as well as gradient related methods subject to angle and scaling conditions. In addition,~in Theorem~\ref{thm: inexact gradient} below we show how~\eqref{eq: estimate for function values} already implies the sharp rates for inexact gradient methods under relative error bounds with fixed step sizes, based on a suitable change of the metric, thereby improving and simplifying a similar result in~\cite{deKlerk2020}.

\section{Variable metric method}

We first consider the variable metric method. Here the update rule reads
\begin{equation}\label{eq: variable metric method}
 x^+ = x - h A^{-1} \nabla f(x),
\end{equation}
where $A$ is a symmetric (with respect to the given inner product) and positive definite matrix. It is well known that such an update step can also be interpreted as a gradient step with respect to a modified inner product. This leads to the following result that will be the basis for our further considerations.

\begin{theorem}\label{thm: variable metric method}
Assume the eigenvalues of $A$ are in the positive interval $[\lambda,\Lambda]$ and define
\[
 \bar h = \frac{2}{L/\lambda + \mu/\Lambda}.
\]
Then $x^+$ in~\eqref{eq: variable metric method} with $0 \le h \le \bar h$ satisfies
\[
 f(x^+) - f(x^*) \le \left(1 - \frac{h \mu}{\Lambda} \right)^2 (f(x) - f(x^*)).
\]
In particular, the step size $h = \bar h$ yields
\begin{equation}\label{eq: estimate for function values variable metric}
 f(x^+) - f(x^*) \le \left(\frac{\kappa_{f,A} - 1}{\kappa_{f,A} + 1} \right)^2 (f(x) - f(x^*)), \quad \kappa_{f,A} = \frac{L }{\mu} \, \frac{\Lambda }{\lambda}. 
\end{equation}
\end{theorem}

\begin{proof}
The result is obtained from~\eqref{eq: function values general h} by noting that $\nabla_A f(x) =   A^{-1}\nabla f(x)$ is the gradient of $f$ with respect to the $A$-inner product $\langle x, y \rangle_A = \langle x, Ay \rangle$. We have
\[
\langle \nabla_A f(x) - \nabla_A f(y), x - y \rangle_A \le L \| x - y \|^2 \le \frac{L}{\lambda} \| x - y \|_A^2
\]
as well as
\[
\langle \nabla_A f(x) - \nabla_A f(y), x - y \rangle_A \ge \mu \| x - y \|^2 \ge \frac{\mu}{\Lambda} \| x - y \|_A^2
\]
for all $x,y$. These two conditions are equivalent to $f$ being $(L/\lambda)$-smooth and $(\mu/\Lambda)$-strongly convex in that $A$-inner product; see, e.g.,~\cite[Theorems~2.1.5~\&~2.1.9]{Nesterov2004}. Thus in~\eqref{eq: feasible step sizes} and~\eqref{eq: function values general h}, we can replace $\mu$ with $\mu/\Lambda$ and $L$ by $L/\lambda$, which is exactly the statement of the theorem.
\end{proof}
An alternative, and somewhat more direct proof of Theorem~\ref{thm: variable metric method} that does not require changing the inner product, can be given by applying the result~\eqref{eq: function values general h} directly to the function $g(y) = f(A^{-1/2}y)$ at $y = A^{1/2}x$.

Observe that $\kappa_{f,A} = \kappa_f \cdot \kappa_A$ with $\kappa_A = \Lambda / \lambda \ge 1$ the condition number of $A$. The contraction factor in~\eqref{eq: estimate for function values variable metric} will therefore always be worse than the original factor in~\eqref{eq: estimate for function values}, which corresponds to $A=I$. This might seem suboptimal since in Newton's method, and under additional regularity conditions, the contraction factor improves when choosing $A = \nabla^2 f(x)$. However, for the general class of methods~\eqref{eq: variable metric method}, the result in Theorem~\ref{thm: variable metric method} is optimal. This can already be seen for the function $f(x) = \frac{1}{2} \| x\|^2$, in which case~\eqref{eq: variable metric method} becomes the linear iteration $x^+ = (I - hA^{-1})x$. Its contraction factor as predicted by~\eqref{eq: estimate for function values variable metric} is bounded by $(\kappa_A - 1)^2/(\kappa_A + 1)^2$, which is indeed a tight bound: as in \cite[Example~1.3]{deKlerk2017}, take $A = \diag(\lambda, \ldots, \Lambda)$ and $x = (x_1, 0, \ldots, 0, x_n)$. Then an exact line search yields $x^+ = (\kappa_A - 1)/(\kappa_A + 1) \cdot (-x_1, 0, \ldots, 0, x_n)$, and clearly there cannot be a better contraction factor with fixed step size. Note that the step size $\bar h$ in Theorem~\ref{thm: variable metric method} also leads to equality in~\eqref{eq: estimate for function values variable metric} when $x$ is an eigenvector corresponding to $\lambda$ or $\Lambda$. For a less trivial example, consider $f(x) = \frac{1}{2} \langle x, A^{-1} x \rangle$. Then~\eqref{eq: variable metric method} becomes $x^+ = (I - hA^{-2})x$ and the same $x$ from above now leads to a contraction with the factor $(\kappa_{A^2} - 1)^2/(\kappa_{A^2}+ 1)^2$ where indeed $\kappa_{A^2} = \kappa_f \kappa_A$, as predicted by Theorem~\ref{thm: variable metric method}.

\section{Gradient related methods}

Next we provide error estimates for gradient related descent methods under angle and scaling conditions. Specifically, we consider the update rule
\begin{equation}\label{eq: perturbed GD}
 x^+ = x - h d,
\end{equation}
where $-d$ is a descent direction, that is, $d$ satisfies
\begin{equation}\label{eq:def of cos theta}
\langle \nabla f(x), d \rangle = \cos \theta \| \nabla f(x) \| \| d \|, \quad \cos \theta > 0,
\end{equation}
for some $\theta \in [0, \pi/2)$. This condition is very natural since it guarantees the convergence of~\eqref{eq: perturbed GD}; see, e.g., \cite[Chapter 3.2]{nocedalNumericalOptimization2006}. In particular, for  the case of exact line search, it has been shown in~\cite[Theorem~5.1]{deKlerk2017} that
\begin{equation}\label{eq:optimal rate angle condition}
f(x^+) - f(x^*) \le \left(\frac{\kappa_{f,\theta} - 1}{\kappa_{f,\theta} + 1} \right)^2 (f(x) - f(x^*)), \quad \kappa_{f,\theta} = \frac{L}{\mu} \left(\frac{1 + \sin \theta}{1 - \sin \theta}\right),
\end{equation}
and that this Q-linear rate is optimal. For the case of quadratic functions this has been known before; see, e.g.,~\cite{MuntheKaas1987}. We also mention the result of~\cite[Theorem~3.3]{Cohen1981}, which identifies the rate in~\eqref{eq:optimal rate angle condition} as optimal $R$-linear rate for exact line search when $f$ is twice continuously differentiable. 

Here, we aim to generalize this result to fixed step sizes. The extent to which this is possible depends on the available information about the quantities $\| \nabla f(x) \|$, $\| d \|$, and $\langle \nabla f(x), d \rangle$. The basic idea is to interpret~\eqref{eq: perturbed GD} as a variable metric method in order to apply Theorem~\ref{thm: variable metric method}. For this we need to find a symmetric and positive definite matrix $A$ satisfying
\[
A d = \nabla f(x)
\]
and estimate its condition number. Such a matrix can be found explicitly using the following lemma, which originates from the SR1 update rule; see, e.g.,~\cite{nocedalNumericalOptimization2006}.

\begin{lemma}\label{lem: SR1 update}
Let $u, v \in \R^n$ such that $\|u \| = \| v \| = 1$ and $\langle u, v \rangle = \cos \theta$. Then the matrix 
\[
 B = \frac{1}{\alpha} \left( I - \frac{r r^*}{\langle r, u \rangle} \right), \quad r = u - \alpha v, \quad \alpha = \frac{1 - \sin \theta}{\cos \theta} = \frac{\cos \theta}{1+\sin \theta}
\]
is symmetric (for the given inner product), satisfies $Bu=v$, and has
\[
\lambda_{\min}(B) = \frac{\cos \theta}{1 + \sin \theta}, \quad \lambda_{\max}(B) = \frac{\cos \theta}{1 - \sin \theta},
\]
as its smallest and largest eigenvalues, respectively. Here, $rr^*$ denotes the rank-one matrix satisfying $rr^* x = r \langle r,x\rangle$ for all $x \in \R^n$.
\end{lemma}

\begin{proof}
This is checked by a straightforward calculation. Obviously, the matrix $I - \frac{\ r^{} r^*}{\langle r, u \rangle}$ equals the identity on the orthogonal complement of $r$. Its eigenvalue belonging to the  eigenvector $r$ is
\[
1 - \frac{\| r \|^2}{\langle r, u \rangle} = 1 - \frac{1 - 2 \alpha \cos \theta + \alpha^2}{1 - \alpha \cos \theta} = \frac{1 - \sin \theta - \alpha^2}{\sin \theta} = \alpha^2,
\]
where one uses $1 - \alpha \cos \theta = \sin \theta$ and $\alpha^2 = (1 - \sin \theta)/(1 + \sin \theta)$. Therefore, the largest eigenvalue of $B$ is $1/\alpha$ (with multiplicity $n-1$), and the smallest eigenvalue is $\alpha$.
\end{proof}

With Lemma~\ref{lem: SR1 update} and Theorem~\ref{thm: variable metric method} at our disposal, we can state our main result.

\begin{theorem}\label{thm: gradient related}
Assume~\eqref{eq:def of cos theta} and 
\begin{equation}\label{eq: scaling condition}
 \| d \| = c \| \nabla f(x) \|
\end{equation}
for some $c > 0$. Define
\[
 \bar h = \frac{2 \cos \theta}{L c (1 + \sin \theta) + \mu c (1 - \sin \theta)}.
\]
Then $x^+$ in~\eqref{eq: perturbed GD} with $0 \le h \le \bar h$ satisfies
\[
 f(x^+) - f(x^*) \le \left(1 - \frac{h \mu c (1-\sin \theta)}{ \cos \theta} \right)^2 (f(x) - f(x^*)).
\]
In particular, the step size $h=\bar h$ yields
\[
 f(x^+) - f(x^*) \le \left(\frac{\kappa_{f,\theta} - 1}{\kappa_{f,\theta} + 1} \right)^2 (f(x) - f(x^*)).
\]
\end{theorem} 

\begin{proof}
If $d = 0$, the assertion is trivially true. Let $d \neq 0$. By Lemma~\ref{lem: SR1 update}, there exists a symmetric and positive definite matrix of the form $A = \frac{\| \nabla f(x) \|}{\| d \|}B = \frac{1}{c}B$ such that $Ad = \nabla f(x)$ and
\[
\lambda_{\min}(A) = \frac{1}{c} \left(\frac{\cos \theta}{1 + \sin \theta} \right), \quad \lambda_{\max}(A) = \frac{1}{c} \left(\frac{\cos \theta}{1 - \sin \theta} \right).
\]
The assertion follows therefore directly from Theorem~\ref{thm: variable metric method}.
\end{proof}

\begin{remark}
The condition~\eqref{eq: scaling condition} can be replaced with equivalent conditions such as
\[
\langle \nabla f(x), d \rangle = \sigma \| d \|^2
\]
for some $\sigma > 0$. An equivalent version of Theorem~\ref{thm: gradient related} is obtained by observing that $\cos \theta = \sigma c$.
\end{remark}

To achieve the optimal rate in Theorem~\ref{thm: gradient related}, the exact values of $\theta$ and $c$ need to be known in order to compute the optimal step size $\bar h$. In practice, this is almost never the case and only bounds are available. We therefore formulate another, more practical result of~\eqref{eq:def of cos theta} under the following relaxed angle and scaling conditions: there exists $0 < c_1 \le c_2$ and $\theta' \in [0,\pi/2)$ such that
\begin{equation}\label{eq: conditions}
\theta \le \theta', \quad c_1 \| \nabla f(x) \| \le \| d \| \le c_2 \| \nabla f(x) \|.
\end{equation}
Under these conditions, the eigenvalues of the matrix $A = \frac{\| \nabla f(x) \|}{\| d \|}B$ in the proof of Theorem~\ref{thm: gradient related} can be bounded as
\[
\lambda_{\min}(A) \ge \frac{1}{c_2} \left(\frac{\cos \theta'}{1 + \sin \theta'} \right), \quad \lambda_{\max}(A) \le \frac{1}{c_1} \left(\frac{\cos \theta'}{1 - \sin \theta'}\right),
\]
since $\cos \theta/(1 \pm \sin \theta)$ is monotonically decreasing/increasing in $\theta \in [0, \pi/2)$. The following result is then again immediately obtained from Theorem~\ref{thm: variable metric method}.

\begin{theorem}\label{thm: gradient related practical}
Assume~\eqref{eq: conditions} and define
\[
 \bar h = \frac{2 \cos \theta' }{L c_2  (1 + \sin \theta') + \mu c_1 (1 - \sin \theta')}.
\]
Then $x^+$ in~\eqref{eq: perturbed GD} with $0 \le h \le \bar h$ satisfies
\[
 f(x^+) - f(x^*) \le \left(1 - \frac{h \mu c_1 ( 1 - \sin \theta')}{ \cos \theta'} \right)^2 (f(x) - f(x^*)).
\]
In particular, the step size $h=\bar h$ yields
\[
 f(x^+) - f(x^*) \le \left(\frac{\kappa' - 1}{\kappa' + 1} \right)^2 (f(x) - f(x^*)), \quad \kappa' =  \frac{L}{\mu} \frac{c_2}{c_1} \left( \frac{1 + \sin \theta'}{1 - \sin \theta'} \right).
\]
\end{theorem}

We remark again that if $c_1 = c_2 = \| d \| / \| \nabla f(x) \|$ and $\theta' = \theta$ are known, the resulting statements from Theorem~\ref{thm: gradient related practical} coincide with those in Theorem~\ref{thm: gradient related}.

\begin{remark}
We conclude the section with a side remark. When just looking at the proofs of Theorems~\ref{thm: gradient related} or~\ref{thm: gradient related practical}, it would be natural to ask if there exists a symmetric and positive definite matrix $B$ (and thus $A$) with a smaller condition number than the one from Lemma~\ref{lem: SR1 update}. As for the SR1 update rule, when matrix $B = B_\alpha$ in the lemma is regarded as a function of $\alpha \neq 0$, then it is well known that the stated $\alpha$ is one of the minimizers for the condition number in the class of all positive definite $B_\alpha$ (another is $\cos \theta / (1 - \sin \theta)$); see, e.g.,~\cite{Wolkowicz1994}. Indeed, any $B$ with a smaller condition number would lead to a faster rate in Theorem~\ref{thm: gradient related} (via Theorem~\ref{thm: variable metric method}), which is not possible since the rate is known to be optimal when exact line search is used. This reasoning therefore provides a (rather indirect) proof for the following general statement.
\end{remark}

\begin{theorem}
Let $u, v \in \R^n$ such that $\|u \| = \| v \| = 1$ and $\cos \theta = \langle u, v \rangle > 0$ with $\theta \in [0,\pi/2)$. Then $(1 + \sin \theta)/(1 - \sin \theta)$ is the minimum possible (spectral) condition number among all symmetric and positive definite matrices $B$ satisfying $Bu = v$.
\end{theorem}

\noindent
While probably well known in the field, we did not find this fact explicitly stated in the literature. It is, of course, not very difficult to prove this result directly by an elementary calculation on $2 \times 2$ matrices.

\section{Inexact gradient method}

We now discuss the important case of an inexact gradient method, where instead of the angle and scaling conditions~\eqref{eq: conditions}, it is assumed that
\begin{equation}\label{eq: inexact gradient condition}
\| d - \nabla f(x) \| \le \varepsilon \| \nabla f(x) \|
\end{equation}
for some $\varepsilon \in [0,1)$. This model is also considered in~\cite{deKlerk2017,deKlerk2020,Gannot2021}. Our aim is again deriving convergence rates for a fixed step size rule from the variable metric approach. Since the matrix $A$ in the proof of Theorem~\ref{thm: gradient related} no longer provides the optimal rates in this case, we use a different construction.

\begin{lemma}\label{lem: matrix A for inexact gradient}
 Let $u,v \in \R^n$ such that $v \neq 0$ and $\| u - v \| < \| v \|$. There exists a positive definite matrix $A$ that satisfies $Au = v$ and has eigenvalues $\left( 1 \pm  \frac{\| u -v \|}{\| v \|} \right)^{-1}$. 
\end{lemma}

\begin{proof}
 Define
 \(
A^{-1} = I + \frac{\| u - v \|}{\| v \|} Q
\)
with $Q = I - 2 \frac{w w^*}{\|w\|^2}$ and $w = \frac{v}{\|v\|} - \frac{u-v}{\| u - v \|}$. Observe that $Q$ is the orthogonal reflection matrix that sends $\frac{v}{\| v \|}$ to $\frac{u - v}{\| u - v \|}$, which implies $A^{-1}v = u$. Since $Q$ is symmetric with eigenvalues $\pm 1$, the result follows.
\end{proof}

Applying the lemma to $u = d$ and $v = \nabla f(x)$, the following theorem on the inexact gradient model~\eqref{eq: inexact gradient condition} is an immediate consequence of Theorem~\ref{thm: variable metric method}.

\begin{theorem}\label{thm: inexact gradient}
Assume $\nabla f(x) \neq 0$ and~\eqref{eq: inexact gradient condition} for some $\varepsilon \in [0,1)$ and define
\[
 \bar h = \frac{2 }{L (1 + \varepsilon) + \mu (1 - \varepsilon)}.
\]
Then $x^+ = x - hd$ with $0 \le h \le \bar h$ satisfies
\begin{align*}
 f(x^+) - f(x^*) &\le \left(1 - h \mu \left( 1 - \frac{\| d - \nabla f(x) \|}{ \| \nabla f(x) \|} \right)  \right)^2 (f(x) - f(x^*)) \\
 &\le \left(1 - h \mu ( 1 - \varepsilon ) \right)^2 (f(x) - f(x^*)).
\end{align*}
In particular, the step size $h=\bar h$ yields
\begin{equation}\label{eq: epsilon rate}
 f(x^+) - f(x^*) \le \left(\frac{\kappa_\varepsilon - 1}{\kappa_\varepsilon + 1} \right)^2 (f(x) - f(x^*)), \quad \kappa_\varepsilon =  \frac{L}{\mu} \left( \frac{1 + \varepsilon}{1 - \varepsilon} \right).
\end{equation}
\end{theorem}

The rate in~\eqref{eq: epsilon rate} is optimal under the general assumption~\eqref{eq: inexact gradient condition}, in particular for quadratic $f$ and $d$ satisfying $\langle \nabla f(x), d \rangle = \cos \theta \| d \| \| \nabla f(x) \|$ with $\sin \theta = \varepsilon$. Trivially, for $f(x) = \frac{1}{2}\| x \|^2$ the estimate~\eqref{eq: epsilon rate} is sharp for all $d$ satisfying~\eqref{eq: inexact gradient condition}.

The result in Theorem~\ref{thm: inexact gradient} is not new. 
In~\cite[Proposition~1.5]{Gannot2021} it has been shown that $\left(\frac{\kappa_\varepsilon - 1}{\kappa_\varepsilon + 1} \right)^2$ is an upper bound for the $R$-linear convergence rate of the inexact gradient method with fixed step size $\bar h$. According to~\cite[Remark~1.6]{Gannot2021}, the estimate~\eqref{eq: epsilon rate} per step is implicitly contained in the proof of~\cite[Theorem~5.3]{deKlerk2020}, which, however, is rather technical. In addition, the statement of~\cite[Theorem~5.3]{deKlerk2020} itself covers the rate~\eqref{eq: epsilon rate} only for a range $\varepsilon \in [0, \bar{\varepsilon}]$ with some $\bar{\varepsilon} < \frac{2\mu}{L+\mu}$. Our proof via Lemma~\ref{lem: matrix A for inexact gradient} provides a simple alternative for obtaining the result for all $\varepsilon \in [0, 1)$ directly from the estimate~\eqref{eq: estimate for function values} for the gradient method (which coincides with~\cite[Theorem~5.3]{deKlerk2020} when $\varepsilon=0$).

\section{Conclusions}

Based on the result~\eqref{eq: estimate for function values} due to~\cite{Taylor2018}, we have derived optimal convergence rates for the function values in  gradient related descent methods and inexact gradient methods with fixed step sizes for smooth and strongly convex functions. The results are obtained using an elementary variable metric approach, in which a single step is interpreted as a standard gradient step. This is possible since function values are a metric independent error measure. Compared to existing results, our proofs offer a more direct way for obtaining the convergence rate estimates of perturbed gradient methods given the rates of their exact counterpart.

\end{document}